\documentclass[11pt,letterpaper,reqno]{amsart}
\usepackage[left=30mm,top=30mm,right=30mm,bottom=30mm]{geometry}
\usepackage{amsthm}
\usepackage{graphicx}
\usepackage{amssymb}
\usepackage{stackrel}
\usepackage{comment}
\usepackage[all]{xy}
\usepackage{color}
\usepackage[dvipsnames]{xcolor} 
\usepackage{hyperref}
\hypersetup{
    colorlinks=true,       
    linkcolor=blue,          
    citecolor=blue,        
    filecolor=blue,      
    urlcolor=blue           
    }
\usepackage[nameinlink,capitalize,noabbrev]{cleveref}

\newtheorem{theorem}{Theorem}
\newtheorem*{theorem*}{Theorem}
\newtheorem{corollary}[theorem]{Corollary}
\newtheorem{lemma}[theorem]{Lemma}
\newtheorem{proposition}[theorem]{Proposition}

\newtheorem*{claim*}{Claim}

{\theoremstyle{remark}
 \newtheorem{remark}[theorem]{Remark}}
 \newtheorem*{remark*}{Remark}
{\theoremstyle{definition}

 \newtheorem{example}[theorem]{Example}

}

\begin{document}

\title{Codimension two torus actions on the affine space}

\author[A. Liendo]{Alvaro Liendo} %
\address{Instituto de Matem\'aticas, Universidad de Talca, Talca, Chile} %
\email{aliendo@utalca.cl}

\author[C. Petitjean]{Charlie Petitjean} %
\address{Université de Bourgogne} %
\email{charlie.petitjean@u-bourgogne.fr}

\date{\today}

\thanks{
 \mbox{\hspace{11pt}}{\it Key words}: Linearization conjecture, affine varieties, contractible varieties, torus actions.\\
 \mbox{\hspace{11pt}} The first author was partially supported by Fondecyt project 1240101.}

\begin{abstract}
In this paper, we classify smooth, contractible affine varieties equipped with faithful torus actions of complexity two, having a unique fixed point and a two-dimensional algebraic quotient isomorphic to a toric blow-up of a toric surface. These varieties are of particular interest as they represent the simplest candidates for potential counterexamples to the linearization conjecture in affine geometry.
\end{abstract}

\maketitle

\section*{Introduction}

The linearization conjecture, a major open problem in affine geometry \cite{K96} first proposed by Kambayashi \cite{K79}, is stated as follows:
\begin{quote}
Let $G$ be a linearly reductive algebraic group acting regularly on $\mathbb{A}^{n}$. Then, $G$ has a fixed point, say $P$, and the action of $G$ is linear with respect to a suitable coordinate system of $\mathbb{A}^{n}$ with $P$ as the origin.
\end{quote}

In the two-dimensional affine space, Jung \cite{J42} and Gutwirth \cite{G62} proved the conjecture prior to Kambayashi's formal statement, and their work served as a precursor to the actual conjecture. 

For the affine spaces of dimention 3 and 4, Kraft, Popov, and Panyushev \cite{KP85,P83} established that every semisimple group action is linearizable. Nevertheless, Schwarz \cite{S89} disproved this conjecture by providing examples of non-linearizable actions of the orthogonal group $\operatorname{O}_{2}$ on $\mathbb{A}^{4}$ and of $\operatorname{SL}_{2}$ on $\mathbb{A}^{7}$. To date, counterexamples have been found for all connected reductive groups except tori \cite{K91}, as well as for many non-commutative finite groups \cite{MMJP91}.

\smallskip

We restrict ourselves now to the case of tori. Let $X$ be an algebraic variety of dimension $n$ over the complex numbers $\mathbb{C}$. Let $T=\mathbb{G}_\textrm{m}^{d}$ be the algebraic torus of dimension $d$. The complexity of a $T$-action $\alpha$ on a variety $X$ is the codimension of the general orbit. If $\alpha$ is faithful, the complexity corresponds simply to $\dim(X) - \dim(T)=n-d$. In the case where the $T$-action has complexity 0, we say that $X$ is a toric variety.

Białynicki-Birula  showed that any codimension-one torus action in the affine space is  linearizable \cite{BB67}. And after a long series of papers by Kambayashi, Koras, Russell, Kaliman, and Makar-Limanov proved that the  linearization conjecture holds for 1-dimensional torus acting on 3-dimensional affine space with complexity 2. \cite{KR82,KR86,KR88.1,KR88.2,KR97,KML97,KKMLR97}. 

Let us now review the state of the art for complexity-two torus actions on affine space.  Independently, Koras-Russell \cite{KR88.1} and Kraft-Schwarz \cite{KS88,KS92} proved that if the fixed-point locus $F$ has dimension at least one, then the action is linearizable. The latter two authors, applying results by Bass-Haboush \cite{BH85}, showed that if the algebraic quotient $\mathbb{A}^n/\!/T$ is 1-dimensional, then it is isomorphic to the affine line $\mathbb{A}^1$, and moreover, that the set of fixed points is either a point or isomorphic to $\mathbb{A}^1$. Finally, in the case where the action is unmixed, i.e., where the algebraic quotient is $0$-dimensional and provides a retraction to the fixed point set, the action is once again linearizable by \cite{KR82}. 

Summarizing, in the case of complexity-two torus actions, the only remaining domain where counterexamples to the linearization conjecture may be found is for a torus action with an unique fixed point and an a $2$-dimensional algebraic quotient. Moreover, this algebraic quotient must be isomorphic to $\mathbb{A}^{2}/\!/\mu$, where $\mu$ is a finite cyclic group by \cite{GKR08}. 

The aim of this paper is to prove the following theorem, which is a generalization of \cite[Theorem 4.1]{KR97} and \cite[Theorem 8]{P18}.

\begin{theorem}
\label{thm:A-smooth-contractible}
A smooth, contractible affine variety $X$ with a faithful $T$-action of complexity two, a unique fixed point $x_0$, and an algebraic quotient isomorphic to $\mathbb{A}^{2}/\!/\mu$, where $\mu$ is a finite cyclic group, is determined by the following data:
\begin{enumerate}
\item[$(a)$] The linear $T$-action on the tangent space $T_{x_0}X \simeq \mathbb{A}^{n}$ at the fixed point $x_0$. 
\item[$(b)$] Two curves $C_1$ and $C_2$ in $\mathbb{A}^2$, invariant under $\mu$, each containing the origin and isomorphic to $\mathbb{A}^1$, intersect with only simple normal crossings.
\end{enumerate}
\end{theorem}

Recall that two curves $C_1, C_2$ in  $\mathbb{A}^2$ are said to intersect only with simple normal crossing if, for every $x\in C_1\cap C_2$, there exists an analytic neighborhood $U$ of $x\subset \mathbb{A}^2$ such that $(U,C_1\cap U,C_2\cap U)$ is analytically isomorphic to $\mathbb{A}^2$ with the coordinate axes.

\section{Overview of the Altmann-Hausen framework}
\label{recollection}

This section provides a brief overview of the  combinatorial description of $T$-varieties, which will be used throughout the paper. A $T$-variety is a normal variety endowed with a faithful regular action of an algebraic torus $T$. The complexity of a $T$-variety is determined as the codimension of a general orbit under the $T$-action. Since the action is faithful, the complexity of a $T$-variety $X$ is given by $\dim X - \dim T$. The simplest examples of $T$-varieties are those of complexity zero, commonly referred to as toric varieties. These were initially introduced by Demazure in \cite{D70}. Toric varieties are characterized by their combinatorial structure, which is encoded in certain collections of strongly convex rational polyhedral cones in $\mathbb{R}^n$, known as fans.For further details, refer to \cite{CLS11}.

For $T$-varieties of higher complexity, a geometric-combinatorial description is also available. Altmann and Hausen developed such a description for affine $T$-varieties in \cite{AH06}, which we refer to as the AH-presentation. Below, we provide a concise exposition of this presentation.

An algebraic torus $T$ can be described via its character and $1$-parameter subgroup lattices, denoted $M$ and $N$, respectively. The group $M$ is composed by characters, i.e., morphisms $u: T \to \mathbb{C}^*$, while $N$ consists of $1$-parameter subgroups, i.e., morphisms $p: \mathbb{C}^* \to T$. Both $M$ and $N$ are free abelian groups of rank $\dim T$, with a natural pairing $M \times N \to \mathbb{Z}$ given by $\langle u, p \rangle$, where $\langle u, p \rangle$ satisfies $u \circ p(t) = t^{\langle u, p \rangle}$ for all $t \in \mathbb{C}^*$. We also let $M_\mathbb{R}=M\otimes_\mathbb{Z} \mathbb{R}$ and $N_\mathbb{R}=N\otimes_\mathbb{Z} \mathbb{R}$ be the corresponding real vector spaces and we extend the duality to them.

The combinatorial nature of $T$-varieties derives from the well known fact that the comorphism $\mathbb{C}[X]\rightarrow \mathbb{C}[M]\otimes \mathbb{C}[X]$ of the $T$-action on an affine $T$-variety induces an $M$-grading on $\mathbb{C}[X]$ and, conversely, every $M$-grading on $\mathbb{C}[X]$ arises in this way
\cite[Section~2.1]{KR82}.

\subsection{Toric varieties}
\label{sec:toric}

A toric variety is a $T$-variety of complexity zero, meaning it is a normal variety with a faithful $T$-action and an open dense orbit. Its combinatorial structure is captured by a fan $\Sigma$ in $N_\mathbb{R}$, which is a finite collection of strictly convex rational polyhedral cones closed under taking faces and intersections. For any cone $\sigma \in \Sigma$, the associated affine toric variety is $X(\sigma) = \operatorname{Spec} \mathbb{C}[\sigma^\vee \cap M]$, where $\sigma^\vee$ is the dual cone in $M_\mathbb{R}$ and $\mathbb{C}[\sigma^\vee \cap M]$ is the semigroup algebra:
\begin{align*}
\mathbb{C}[\sigma^\vee \cap M] = \bigoplus_{u \in \sigma^\vee \cap M} \mathbb{C} \cdot \chi^u,\quad \chi^{u+u'} = \chi^u \cdot \chi^{u'}, \; \forall u, u' \in \sigma^\vee \cap M.
\end{align*}
Furthermore, if $\tau\subseteq\sigma$ is a face of $\sigma$, then the inclusion of algebras $\mathbb{C}[\sigma^{\vee}\cap M]\hookrightarrow \mathbb{C}[\tau^{\vee}\cap M]$ induces a $T$-equivariant open embedding $X(\tau)\hookrightarrow X(\sigma)$ of affine $T$-varieties.  We define the toric variety $X(\Sigma)$ associated to the fan $\Sigma$ as the variety obtained by gluing the family $\{X(\sigma)\mid\sigma\in\Sigma\}$ along the open embeddings $X(\sigma)\hookleftarrow X(\sigma\cap\sigma')\hookrightarrow X(\sigma')$ for all $\sigma,\sigma'\in\Sigma$.

\subsection{Affine $T$-varieties and p-divisors}
\label{sec:comb-des-affine}

Let $\sigma$ be a stronlgly convex rational polyhedral cone in $N_{\mathbb{R}}$. We let
$\operatorname{Pol}_{\sigma}(N_{\mathbb{R}})$  be the semigroup of all rational polyhedra
with tail cone $\sigma$ under Minkowski sum.  Let $Y$ be a normal semiprojective
variety, i.e., $Y$ is projective over an affine variety. A polyhedral divisor on $Y$ is a formal sum
$\mathcal{D}=\sum_{Z}\Delta_Z\cdot Z$, where the sum runs trough all prime divisors
$Z$ in $Y$, $\Delta_Z\in\operatorname{Pol}_{\sigma}(N_{\mathbb{R}})$, and $\Delta_Z=\sigma$ for
all but finitely many prime divisors $Z$. We say that $\sigma$ is the
tail cone of $\mathcal{D}$. For any $u\in\sigma^\vee\cap M$ we can evaluate
$\mathcal{D}$ in $u$ by letting $\mathcal{D}(u)$ be the $\mathbb{Q}$-divisor in $Y$
given by
$$\mathcal{D}(u)=\sum_{Z\subseteq Y} \min\langle u,\Delta_Z\rangle \cdot
Z\,.$$

A polyhedral divisor $\mathcal{D}$ on $Y$ is called a p-divisor if  $\mathcal{D}(u)$
is semiample for all $u\in \sigma^\vee\cap M$ and big for all $u\in M$ contained in the relative interior of $\sigma^\vee$. The AH-presentation of affine $T$-varieties is given in the following theorem \cite{AH06}.

\begin{theorem} \label{T-vars-affine}
To any p-divisor $\mathcal{D}$ on a normal semiprojective variety $Y$ one can associate an affine $T$-variety of complexity $\dim Y$ given by $X(Y,\mathcal{D})=\operatorname{Spec} A(Y,\mathcal{D})$, where  $$A(Y,\mathcal{D})=\bigoplus_{u\in\sigma^{\vee}\cap M} A_u\chi^u,\quad\mbox{and}\quad A_u=H^0(Y,\mathcal{O}(\mathcal{D}(u))\,.$$ %
Conversely, any affine $T$-variety is equivariantly isomorphic to $X(Y,\mathcal{D})$ for some p-divisor $\mathcal{D}$ on some normal semiprojective variety $Y$.
\end{theorem}

In this paper, we will only be interested in \emph{fully hyperbolic} torus actions on $X=X(Y,\mathcal{D})$. This corresponds to the case where $\sigma=\{0\}\in N_\mathbb{R}$ and there is a projective birational morphism $q\colon Y\to Y_0$ where $Y_0= \operatorname{Spec} A_0$ is the spectrum of the ring of $T$-invariant regular functions on $X$. Moreover, we denote by $\pi\colon X\to Y_0$ the natural map comming from the embedding $A_0\hookrightarrow A(Y,\mathcal{D})$.

Given a $T$-variety $X$, a method to determine a p-divisor $\mathcal{D}$ such
that $X=X(Y,\mathcal{D})$ is given in \cite[section 11]{AH06}. The computation is reduced to the toric case by considering a $T$-equivariant embedding $X\hookrightarrow\mathbb{A}^{n}$, where $\mathbb{A}^{n}$ is consider with its unique faithful $\mathbb{G}_m^n$-action up to conjugation. More precisely, realizing $T$ as a subtorus $\mathbb{G}_m^n$, we obtain an inclusion of lattices providing the following exact sequence.
\begin{align} \label{eq:AH}    
\xymatrix{0\ar[r] & N\simeq\mathbb{Z}^{k}\ar[r]_{F} & N'\simeq\mathbb{Z}^{n}\ar[r]_{P}\ar@/_{1pc}/[l]_{s} & N'/N\ar[r] & 0}
,
\end{align} 
where $F$ is the weight matrix is given by the induced action of $T$ on $\mathbb{A}^{n}$
and $s$ is a section of $F$. The quotient toric variety is determined by coarsest fan refining all the first integral vectors
$v_{i}$ of the unidimensional cone generated by the $i$-th column vector of $P$ considered as rays in a $\mathbb{Z}^{n}$ lattice, and each $v_{i}$ correspond to a divisor. The support of $D_{i}$ is the intersection between $X$ and the divisor corresponding to $v_{i}$. The tail cone is $\sigma:=s(\mathbb{Q}_{\geq0}^{n}\cap F(\mathbb{Q}))$, and the polytopes are $\Delta_{i}=s(\mathbb{R}_{\geq0}^{n}\cap P^{-1}(v_{i}))$, see \cite[section 11]{AH06} for further details.

\section{Proof of \cref{thm:A-smooth-contractible}}

We first compute the AH-presentation of a fully hyperbolic linear $T$-action on the affine space. 

\begin{proposition} \label{prop:affine}
Let $X=\mathbb{A}^n$ be endowed with a linear fully hyperbolic $T$-action $\alpha$ of complexity 2 with 2-dimensional algebraic quotient with weight matrix $F$. Then $X$ is equivariantly isomorphic to $X(Y,\mathcal{D})$ where
$Y$ is a blow-up of $\mathbb{A}^{2}/\!/\mu$ and $$\mathcal{D}=\Delta_{1}\otimes H_{1}+\Delta_{2}\otimes H_{2}+\stackrel[i=3]{n}{\sum}\Delta_{i}\otimes E_{i}\,,$$
where $H_1$ and $H_2$ are the strict transform of the images of the coordinate axes of $\mathbb{A}^2$ in $Y$, $E_i$ are the exceptional divisors of $Y$, and the polyhedral coefficients $\Delta_i$ are obtained from $F$ using \eqref{eq:AH}.
\end{proposition}

\begin{proof}
the AH-presentation of $\mathbb{A}^{n}$ endowed with linear fully hyperbolic $T$-action via the exact sequence given in \eqref{eq:AH}. The AH quotient $Y$ is a toric surface since the $T$-action is linear. Moreover, the rays of the fan $\Sigma$ of $Y$ are given by the column vectors of the matrix $P$. Since the algebraic quotient is $2$-dimensional, we have that $Y$ is a blow-up of $\mathbb{A}^{2}/\!/\mu$. The number of rays of $\Sigma$ is bounded by $n$ the number of columns of the matrix $P$. Hence, the number of exceptional divisors is bounded by $n-2$.\\
\end{proof}

The number of blow-up can be arbitrary between $0$ and $n-2$ as illustrated by the following examples.

\begin{example} \label{example-of-blox-up}
\begin{enumerate}
\item[$(i)$]   Let $X=\mathbb{A}^{4}$ endowed with the $T$-action given by the weight matrix 
$$F=\left(\begin{array}{cc}
1 &  0\\
-1 & 0\\
0 & 1\\
0 & -1\\
\end{array}\right)\quad\mbox{with}\quad P=\left(\begin{array}{cccc}
1 & 1 & 0 & 0\\
0 & 0 & 1 & 1
\end{array}\right).$$ 
Hence, $Y$ is $\mathbb{A}^{2}$.

\item[$(ii)$]  Let $X=\mathbb{A}^{n}$ endowed with the $T$-action given by the weight matrix  $$F=\left(\begin{array}{cccc}
1 & \cdots & \cdots & 1 \\
1 & \cdots & \cdots & 1 \\
 -1 & 0 & \cdots & 0\\
 0 & -1 & \ddots & \vdots\\
 \vdots & \ddots & \ddots & 0\\
 0 & \cdots & 0 & -1
\end{array}\right)\quad\mbox{with}\quad P=\left(\begin{array}{ccccc}
1 & 0 & 1 & \cdots & 1\\
0 & 1 & 1 & \cdots & 1
\end{array}\right).$$ 
Hence, $Y$ is $\mathbb{A}^{2}$ blown-up at one point.

\item[$(iii)$] Let $X=\mathbb{A}^{n}$ endowed with the $\mathbb{T}$-action given by the weight matrix 
$$F=\left(\begin{array}{cccc}
n-2 & n-1 & \cdots & 1 \\
1 & \cdots & \cdots & 1 \\
 -1 & 0 & \cdots & 0\\
 0 & -1 & \ddots & \vdots\\
 \vdots & \ddots & \ddots & 0\\
 0 & \cdots & 0 & -1
\end{array}\right)\quad\mbox{with}\quad P=\left(\begin{array}{cccccc}
1 & 0 & n-2 & n-1 &\cdots & 1\\
0 & 1 & 1 & 1 & \cdots & 1
\end{array}\right).$$ 
Hence, $Y$ is a blow-up of $\mathbb{A}^{2}$ having 
$n-2$ exceptional divisors.

\end{enumerate}
\end{example}

In the proof of \cref{thm:A-smooth-contractible}, we will apply the following result proven earlier by the same authors of this paper. 

\begin{proposition}[{\cite[Theorem~7]{LP17}}] \label{prop:smooth}
Let $X = X(Y,\mathcal{D})$ be an affine $T$-variety of dimension $n$. Assume that $(Y,\mathcal{D})$ is minimal. Then, $X$ is smooth if and only if for every point $y \in Y_0$ there exists a neighborhood $U \subseteq Y_0$, a linear $T$-action on $\mathbb{A}^n$ given by the combinatorial data $(Y',\mathcal{D}')$ with $\mathcal{D}'$ minimal, an algebraic variety $V$, and étale morphisms $\alpha : V \to Y$ and $\beta : V \to Y'$ such that $\rho^{-1}(U) \subseteq \alpha(V)$ and $\alpha^*(\mathcal{D}) = \beta^*(\mathcal{D}')$. In particular, the combinatorial data $(Y,\mathcal{D})$ is locally isomorphic in the étale topology to the combinatorial data of the affine space endowed with a linear $T$-action.
\end{proposition}

To prove \cref{thm:A-smooth-contractible}, we need the following lemma.

\begin{lemma} \label{lem:fixed-points}
Let $X=X(Y,\mathcal{D})$ a fully hyperbolic action with $\mathcal{D}=\sum_i \Delta_i\otimes D_i$. Let also $T'\subset T$ be a 1-dimensional subtorus corresponding to a 1-dimensional subspace $L$ of $N_\mathbb{R}$. Then $x\in X$ is fixed under the $T'$-action if and only $q_0(x)\in \pi(D_i)$ and $(L+p)\cap \Delta_i$ is an interval of positive length for some $p\in N_\mathbb{R}$.
\end{lemma}

\begin{proof}
Assume first that $T'=T$ so that $\dim T=1$. The polyhedral divisor $\mathcal{D}$ is then given by
\[
\mathcal{D} = \sum_{i=1}^n [a_i, b_i] \otimes D_i\,.
\]
In this case the result follows by \cite[Proof of Proposition~6]{P18}, see also \cite[Theorem 4.18]{FZ03}. Indeed, we have that:
\begin{itemize}
    \item[$(i)$] If $[a_i, b_i]$ has length 0, then the points $x\in X$ such that $q_0(x)\in \pi(D_i)$ have finite isotropy subgroup and so are not fixed points.
    \item[$(ii)$] If $[a_i, b_i]$ has positive length, then the points $x\in X$ such that $q_0(x)\in \pi(D_i)$ have  infinite isotropy subgroup and so are fixed points.
\end{itemize}
Now, in the general case where $T'\subsetneq T$, the result follows simply embedding $X$ equivariantly in $\mathbb{A}^n$ and applying \eqref{eq:AH} to both tori $T$ and $T'$.
\end{proof}

We now state a stronger version of \cref{thm:A-smooth-contractible} in the introduction whose statements was postponed since it requires the details of the AH-presentation. We will prove \cref{thm:Main} and \cref{thm:A-smooth-contractible} will be a direct  consequence.

\begin{theorem}  \label{thm:Main} 
Let $X$ be a smooth, contractible affine variety $X$ with a faithful $T$-action of complexity two, a unique fixed point $x_0$, and an algebraic quotient isomorphic to $\mathbb{A}^{2}/\!/\mu$, where $\mu$ is a finite cyclic group. Then $X=X(Y,\mathcal{D})$, with
$$\mathcal{D}=\Delta_{1}\otimes D_{1}+\Delta_{2}\otimes D_{2}+\stackrel[i=3]{n}{\sum}\Delta_{i}\otimes E_{i}, \quad \mbox{where}$$ 
\begin{enumerate}
\item[$(a)$] The coefficients $\Delta_i$ are the coefficients obtained in \cref{prop:affine} for the linear linear $T$-action on the tangent space $T_{x_0}X \simeq \mathbb{A}^{n}$ at the fixed point $x_0$. 
\item[$(b)$] The divisors $D_1$ and $D_2$ are strict transforms of two curves $C_1$ and $C_2$ in $\mathbb{A}^2$, respectively. The curves $C_1$ and $C_2$ are invariant under $\mu$, each containing the origin and isomorphic to $\mathbb{A}^1$, intersect with only simple normal crossings.
\end{enumerate}
\end{theorem}

\begin{proof}

 The proof will be organized into three steps, each aiming to verify distinct elements of the AH presentation of a $T$-variety. In the first step, we will establish that no prime divisor other than those provided and statement of \cref{thm:Main}, and encoding the étale neighborhood of the fixed point $x_0$ appear in $\mathcal{D}$ with non-trivial coefficient. In the second step, we will demonstrate that the divisor given by $D=D_1+D_2$ is a simple normal crossing (SNC) divisor. Finally, in the third step, we will prove that the irreducible components of the divisor $D$ are isomorphic to the affine line  $\mathbb{A}^1$.

\subsection*{Step one}
By hypothesis, $X$ admits a unique fixed point for the action of $T$, and $X$ is contractible with an unique fixed point $x_0$. We proceed by contradiction assuming that there exists an additional prime divisor $D_0$ in $\mathcal{D}$ with coefficient $\Delta_0$ and assume that $D_0$ does not intersect an étale neighborhood of the image in $Y$ of the fixed point $x_0$.

Assume first that $\Delta_0$ has non empty relative interior and let $L\subset N_\mathbb{R}$ be a 1-dimensional subspace of $N_\mathbb{R}$ such that $L\cap \Delta_0$ is an interval of positive length. Then by \cref{lem:fixed-points}, there are fixed points at every $x\in X$ such that $q(x)=\pi(D_0)$. Since $D_0$ does not intersect an étale neighborhood of the image in $Y$ of the fixed point $x_0$, we obtain that the set of fixed points of $X$ with respect to the 1-dimensional $T'$-action given by $L$ is not connected. This provides a contradiction.

\medskip

Assume now that the coefficient $\Delta_0$ has empty relative interior, then it encodes a cyclic cover $\pi : X \rightarrow X'$ of positive order along a divisor $H$ of $X'$ \cite[Lemma~7]{P15}. For the fixed point to be preserved during the cyclic covering and for the torus action to remain intact, the support of $H$ must contain the fixed point. This leads to a contradiction since $D_0$ does not intersect an étale neighborhood of the image of the fixed point in $Y$.

\subsection*{Step two}  By \cref{prop:smooth}, the curves must be smooth. Indeed, they must be locally isomorphic to toric curves and all toric curves are smooth. Moreover, the fact that $D$ is an SNC divisor follows from \cref{prop:smooth} and \cref{prop:affine} since the AH-presentation affine space endowed with a linear torus action is given by toric divisors, which are always SNC \cite[Theorem~11.2.2]{CLS11}. 

\subsection*{Step three} To conclude this step we distingush two cases.
\begin{enumerate}
\item[$(i)$] If $\Delta_i$ is a point in the lattice $N_\mathbb{R}$, then by \cite[Theorem A]{K93}, the variety $X$ is obtained by a cyclic cover and is contractible. The divisors used for the cyclic covering must be $\mathbb{Z}_k$-acyclic for almost every $k$, meaning they have the $\mathbb{Z}_k$-homology of a point for almost every $k$. $\mathbb{Z}_k$-acyclic for almost every $k$ implies that they are, in fact, $\mathbb{Z}_k$-acyclic. The Abhyankar-Moh-Suzuki Theorem (see \cite{AM75,S74}) now ensures that this is only possible if and only if the smooth curves defined by $C_i \subset \mathbb{A}^{2}$ are two copies of the affine line $\mathbb{A}^1$.\\

\item[$(ii)$] If $\Delta_i$ admits a non-empty relative interior, then $C_i$ is the image of a fixed-point locus for a subtorus $T'$ of $T$ by \cref{lem:fixed-points}. Now, by \cite[Lemma 5.6]{KPR89} we have that  $X/\!/T'\simeq C_i$ is isomorphic to $\mathbb{A}^1$. 
\end{enumerate}
This concludes the proof.
\end{proof}

\begin{corollary}
Every fully hyperbolic $T$-action of complexity two on the affine space are either linearizable or is obtained after an equivariant bi-cyclic covering of affine space with a linear action.
\end{corollary}
\begin{proof}
   By \cref{thm:Main}, $\mathbb{A}^n = X(Y, \mathcal{D})$, where
$$
\mathcal{D} = \Delta_{1} \otimes D_{1} + \Delta_{2} \otimes D_{2} + \sum_{i=3}^{n} \Delta_{i} \otimes E_{i}.
$$
First, we assume that the coefficients $\Delta_{1}$ and $\Delta_{2}$ have non-empty relative interiors.\\
Let   $S_i$   be the minimal linear subspace containing $\Delta_i$ for $i = 1, 2$, by \cref{lem:fixed-points}.
\begin{enumerate}
    \item[$(i)$] If $S_1 \cap S_2 = 0$, then  $D_1 \cap D_2 $ encodes a set of fixed points (including the origin) for a torus $T$ with  $\dim T = \dim S_1 + \dim S_2$. This set is not connected, so this never occurs.
    \item[$(ii)$] If   $S_1 \cap S_2 \neq 0$, then $ D_1 \cap D_2$ encodes a set of fixed points of dimension $1$ for a torus $T$  with  $\dim T = \dim S_1$ , but by \cite[Lemma 5.6]{KPR89} this must be isomorphic to  $\mathbb{A}^1$, so this never occurs.
\end{enumerate}
Therefore, there exists an automorphism of $\mathbb{A}^2$ preserving the origin \cite{K96,J92}, such that we are precisely in the setting of \cref{prop:affine}, and so the action is linear.

Second, assume that a coefficients $\Delta_{1}$ or $\Delta_{2}$ has empty relative interior. Whithout loss of generality, we can assume that $\Delta_{1}$ has empty relative interior. Then there exists a finite group action $\sigma$ on the affine $X(Y, \mathcal{D})$ space such that $X(Y,\mathcal{D})/\!/\sigma$ is equivariantly isomorphic to affine space with a linear action. Indeed, by \cite[Lemma~7]{P15} the quotient $X(Y, \mathcal{D})/\!/\sigma\simeq X(Y',\mathcal{D}')$ is   a $T$-variety, with $\mathcal{D}'= \Delta'_{2} \otimes D_{2} + \sum_{i=3}^{n} \Delta'_{i} \otimes E_{i}$. However $D_{2}$ is the strict transform of a copy of $\mathbb{A}^1$ by \cite{AM75,S74}, we are once again in the setting of \cref{prop:affine}, and so the action is linear. 
\end{proof}

\begin{lemma}\label{lemma-product}
Let $X=\mathbb{A}^n$ endowed with a linear fully hyperbolic $T$-action of complexity two given by the matrix $F$.\\ Assume that the $T$-action does not act trivially on any coordinate of the affine space.
 If $X$ is equivariantly isomorphic to $X(Y,\mathcal{D})$ where $Y$ is the affine plane then $X$ is the product of two varieties of complexity one.
\end{lemma}

\begin{proof}
    The $T$-action on $\mathbb{A}^n=\operatorname{Spec}(\mathbb{C}[x_1,\ldots,x_n])$  is determined by $F$, the weight matrix, and let $P$ be the matrix defined using \eqref{eq:AH}. By assumption, the rays of $P$ generate the cone of the toric variety $Y$. Since $Y$ is affine, it is isomorphic to the algebraic quotient $\mathbb{A}^n/\!/T\simeq\mathbb{A}^{2}/\!/\mu$ \cite{GKR08}. 
    We can assume, up to a lattice automorphism, that $P$ has the following form:
$$P = \left( \begin{array}{cccccc}
a_1 & \cdots & a_k & a_{k+1} & \cdots & a_n \\
0 & \cdots & 0 & b_{k+1} & \cdots & b_n
\end{array} \right),$$
with $b_i \neq 0$ and there exists a unique rational number $q \in \mathbb{Q}$ such that $\frac{a_i}{b_i} = q$ for all $i = k+1, \dots, n$.

This implies, in particular, that both $\prod_{i=1}^{n}x_i^{a_i}$ and $\prod_{i=k+1}^{n}x_i^{b_i}$ are invariant under the $T$-action. Since $b_i = a_i \cdot q$ for all $i = k+1, \dots, n$, we get:
$$
\prod_{i=k+1}^{n} x_i^{a_i q} = \left( \prod_{i=k+1}^{n} x_i^{a_i} \right)^q,
$$
and $\prod_{i=k+1}^{n}x_i^{a_i}$ is also invariant under the $T$-action.\\
This immediately implies that $P$ is, in fact, of the following form:
$$P = \left( \begin{array}{cccccc}
a_1 & \cdots & a_k & 0 & \cdots & 0 \\
0 & \cdots & 0 & b_{k+1} & \cdots & b_n
\end{array} \right).$$
Thus, $Y = \mathbb{A}^2$ and $\mathbb{A}^n = \mathbb{A}^k \times \mathbb{A}^{n-k} = \operatorname{Spec}(\mathbb{C}[x_1, \dots, x_k]) \times \operatorname{Spec}(\mathbb{C}[x_{k+1}, \dots, x_n])$, which is the product of two affine spaces with complexity-one $T$-actions.

\end{proof}

\section{Examples}

In this section, we will illustrate several phenomena that arise when considering smooth contractible $T$-varieties as in \cref{thm:Main}. In our final \cref{ex:final}, we will present  an exotic space that arise from the setup of our \cref{thm:Main}. Recall that an exotic space is an affine variety not isomorphic to $\mathbb{A}^4$ whose underlying manifold structure is

\begin{example} \label{ex:10}        
Considering the first action in \cref{example-of-blox-up}, i.e., $\mathbb{A}^{4}$ endowed with a  $T$-action  given by the matrix of weight $$F=\left(\begin{array}{cc}
1 &  0\\
-1 & 0\\
0 & 1\\
0 & -1\\
\end{array}\right)\quad\mbox{with}\quad P=\left(\begin{array}{cccc}
1 & 1 & 0 & 0\\
0 & 0 & 1 & 1
\end{array}\right),$$  so that $Y$ is isomorphic to $\mathbb{A}^{2}$   and the section $s$ in \eqref{eq:AH} can be chosen as $\left(\begin{array}{cccc}
1 & 0 & 0 & 0\\
0 & 0 & 1 & 0
\end{array}\right)$.

In this example $\mathbb{A}^{4}$ is the product of two copies of $\mathbb{A}^{2}$ endowed with a being a complexity one torus action having an unique fixed point $p_i$, see \cref{lemma-product}. In this case $\mathbb{A}^{4}=X(Y,\mathcal{D})$ is determine by $Y\simeq \mathbb{A}^2=\operatorname{Spec}\:\mathbb{C}[u,v]$  and $\mathcal{D}=([-1,0]\times 0)\otimes\{u=0\}+(0\times[-1,0])\otimes\{v=0\}$. Each prime divisor in $\mathcal{D}$ characterizing an irreducible component of the  fix locus of the $T$-action.
\end{example}

\begin{example}    
We now consider the same quotient in  $Y \simeq \mathbb{A}^2=\operatorname{Spec}\:\mathbb{C}[u,v]$ and polyhedral coefficients as in \cref{ex:10}, but we let  $$\mathcal{D} = ([-1,0] \times 0) \otimes \{f_1(u,v)=0\} + (0 \times [-1,0]) \otimes \{f_2(u,v)=0\}.$$ 
we then obtain the affine variety $X=X(Y,\mathcal{D})$ determined by the affine algebra
$$
A(X,\mathcal{D}) =\mathbb{C}[u,v,x_1,x_2,x_3,x_4]/(x_1x_2 - f_1(u,v), x_3x_4 - f_2(u,v)).
$$
If we specialize the divisors, $\{f_1(u,v)=u\}$  and $\{f_2(u,v)=v\}$, we recover the affine space with a linear action. If we specialize the divisors, $\{f_1(u,v)=u\}$  and $\{f_2(u,v)=u+v+v^2\}$, that are two curves, each isomorphic to $\mathbb{A}^{1}$, but having two intersection points: the origin and also $(0, -1)$. In this case we obtain that  $X=X(Y,\mathcal{D})$ determined by the affine algebra
$$
A(X,\mathcal{D}) = \mathbb{C}[v,x_1,x_2,x_3,x_4]/( x_3x_4 -x_1x_3 -v-v^2).
$$
This variety admits two different fixed points, determined by the intersection of the two curves. These points are given by the intersection of the line $x_1 = x_2 = x_3 = x_4 = 0$, which is fixed under the linear action in the ambient space, and the variety $X$. As a result, we obtain the fixed points $(0, 0, 0, 0, 0)$ and $(-1, 0, 0, 0, 0)$.
\end{example}

\begin{example}
We consider now the example in \cref{example-of-blox-up}~$(ii)$ with $n=4$. Let $X$ be the affine space $\mathbb{A}^{4}$ endowed with a $T$-action given by the matrix of weight $$F=\left(\begin{array}{cc}
1 &  1\\
1 & 1\\
-1 & 0\\
0 & -1\\
\end{array}\right)\quad\mbox{with}\quad P=\left(\begin{array}{ccccc}
1 & 0 & 1 &  1\\
0 & 1 & 1 &  1
\end{array}\right).$$
Then $X=X(Y,\mathcal{D})$, where $Y$ is the blow-up of  $\mathbb{A}^2=\operatorname{Spec}\:\mathbb{C}[u,v]$ at the origin and $\mathcal{D}=\Delta \otimes E$, with $E$ is the exceptional divisor of the blow-up and $\Delta$ is the opposite of the standard $2$-dimensional simplex, i.e., $\Delta$ is the polytope obtained as the convex hull of the points $\{(0,0);(-1,0);(0,-1)\}$. A straightforward generalization shows that $\mathbb{A}^n$ in \cref{example-of-blox-up}~$(ii)$ has the same quotient $Y$ and $\mathcal{D}=\Delta\otimes E$,
where $\Delta$ is the opposite of the standard $n-2$-dimensional simplex.
\end{example}

\begin{example}
Let now $X$ be the affine space  $\mathbb{A}^{4}$  endowed with the $T$-action given by the weight  matrix $$F=\left(\begin{array}{cc}
1 &  1\\
-1 & -1\\
-1 & 0\\
0 & -1\\
\end{array}\right)\quad\mbox{with}\quad P=\left(\begin{array}{ccccc}
1 & 1 & 0 &  0\\
1 & 0 & 1 &  1
\end{array}\right).$$
Now, $X=X(Y,\mathcal{D})$, where $Y$ is the blow-up up of $\mathbb{A}^2=\operatorname{Spec}\:\mathbb{C}[u,v]$ at the origin and $\mathcal{D}=\Delta_1\otimes D_1+ \Delta \otimes E$. Here $D_1$ is the strict transform of $\{u=0\}$, $E$ is the exceptional divisor of the blow-up, $\Delta_1$ is the convex hull of $\{(0,1);(1,0)\}$, and $\Delta$ is the standard 2-dimensional simplex.
\end{example}

\begin{example}\label{ex:14}
Let $X'$ be the affine space $\mathbb{A}^5=\operatorname{Spec}(\mathbb{C}[x,y_1,y_2,z,t])$ endowed with the $T$-action given by the weight matrix
$$F=\left(\begin{array}{cc}
6 &  0\\
-6 & 2\\
0 & -1\\
3 & 0\\
2 & 0\\
\end{array}\right)\quad\mbox{with}\quad P=\left(\begin{array}{cccccc}
1 & 1 & 2 & 0 &  0\\
0 & 1 & 2 & 2 &  0\\
0 & 1 & 2 & 0 &  3
\end{array}\right).$$
In this case, the algebraic quotient is $$\mathbb{A}^{5}/\!/T\simeq\mathbb{A}^{3}=\operatorname{Spec}(\mathbb{C}[z^{2}y_{1}y_{2}^{2},t^{3}y_{1}y_{2}^{2},xy_{1}y_{2}^{2}])=\operatorname{Spec}(\mathbb{C}[w,u,v])\,,$$
and the AH-quotient $Y'$ is the three dimensional affine space $\mathbb{A}^3$ blown-up at the origin. Then, $X'=X(Y',\mathcal{D})$, where  $\mathcal{D}=\Delta_1\otimes D'_1+\Delta_2\otimes D'_2+\Delta\otimes E'$. Here,
\begin{itemize}
\item $\Delta_{1}=(\frac{2}{3},0)$ and $D'_{1}$ is the strict transform of $\{u=0\}.$

\item  $\Delta_{2}=(\frac{-1}{2},0)$ and  $D'_{2}$ is the strict transform of $\{w=0\}.$

\item $\Delta$ is the convex hull of  $\left\{(0,0);(\frac{1}{6},0);(0,-\frac{1}{2})\right\}$  and $E'$ is the exceptional divisor.
\end{itemize}
\end{example}

\begin{example} \label{ex:final}
With the notation of \cref{ex:14}, let $Y_0=\{w+u+v+v^{2}=0\}$ and let $Y$ be its strict transform in $Y'$. A straightforward computation shows that $Y_0\simeq \operatorname{Spec}\mathbb{C}[u,v]$ and $Y$ is the blow-up of the origin in $Y_0$. Now, letting $\mathcal{D}=\Delta_1\otimes D_1+\Delta_2\otimes D_2+\Delta\otimes E$, where $E=E'\cap Y$ is the exceptional divisor and $D_i=Y\cap D'_1$ for $i=1,2$. Remark that $D_2=\{u+v+v^2\}$.

A straightforward computation shows that the affine variety $X=X(Y,\mathcal{D})$ is isomorphic to $\{x+x^2y_1y_2^2+z^2+t^3=0\} \subset \mathbb{A}^5$. Moreover, by \cref{thm:Main}, $X$ is a smooth and contractible having an unique fixed point. The variety $X$ is an exotic space. Indeed, in a preprint, the Makar-Limanov of $X$ been shown to be non-trivial \cite{GGP24}. Since the Makar-Limanov of the affine space is trivial, this shows that $X$ is not isomorphic to the affine space. See \cite{F17,ML96} for details on the Makar-Limanov invariant.
\end{example}

\begin{remark}
The variety $X$ can be viewed as an affine modification and multi-cylic cover of the affine space $\mathbb{A}^5$. By \cite{K93,KZ99}, it follows that the varieties are topologically contractible. It is not know if these varieties are stably affine spaces. Hence, all theses varieties are potential counterexamples to Zariski Cancellation Problem.
\end{remark}


\begin{thebibliography}{xxxxxxxxx}



\bibitem[AM75]{AM75} Abhyankar, Shreeram S.; Moh, Tzuong Tsieng
Embeddings of the line in the plane.
J. Reine Angew. Math.276(1975), 148–166.

\bibitem[AH06]{AH06} Altmann, Klaus; Hausen, Jürgen
Polyhedral divisors and algebraic torus actions.
Math. Ann.334(2006), no.3, 557–607.


\bibitem[BH85]{BH85} Bass, H.; Haboush, W.
Linearizing certain reductive group actions.
Trans. Amer. Math. Soc.292(1985), no.2, 463–482. 


\bibitem[BB67]{BB67}
Białynicki-Birula, A.
Remarks on the action of an algebraic torus on $k^n$. II. 
Bull. Acad. Polon. Sci. Sér. Sci. Math. Astronom. Phys.15(1967), 123–125.

\bibitem[CLS11]{CLS11} Cox, David A.; Little, John B.; Schenck, Henry K.
Toric varieties.
Grad. Stud. Math., 124
American Mathematical Society, Providence, RI, 2011.

\bibitem[D70]{D70} Demazure, Michel
Sous-groupes algébriques de rang maximum du groupe de Cremona.
Ann. Sci. École Norm. Sup. (4)3(1970), 507–588.

\bibitem[FZ03]{FZ03} Flenner, Hubert; Zaidenberg, Mikhail
Normal affine surfaces with $\mathbb{C}^*$-actions.
Osaka J. Math.40(2003), no.4, 981–1009.

\bibitem[F17]{F17} Freudenburg, Gene
Algebraic theory of locally nilpotent derivations.
Second edition
Encyclopaedia Math. Sci., 136
Invariant Theory Algebr. Transform. Groups, VII
Springer-Verlag, Berlin, 2017. xxii+319 pp.
ISBN:978-3-662-55348-0
ISBN:978-3-662-55350-3

\bibitem[GGP24]{GGP24}On embedding of linear hypersurfaces
Parnashree Ghosh, Neena Gupta, Ananya Pal https://arxiv.org/abs/2405.07205




\bibitem[G62]{G62} Gutwirth, A.
The action of an algebraic torus on the affine plane.
Trans. Amer. Math. Soc.105(1962), 407–414.

\bibitem[GKR08]{GKR08} Gurjar, R. V.; Koras, Mariusz; Russell, Peter
Two dimensional quotients of $\mathbb{C}^{n}$ by a reductive group.
Electron. Res. Announc. Math. Sci.15(2008), 62–64. 

\bibitem[J92]{J92} Jelonek, Zbigniew(PL-JAGL)
The Jacobian conjecture and the extensions of polynomial embeddings.
Math. Ann.294(1992), no.2, 289–293.

\bibitem[J42]{J42} Jung, Heinrich W. E.
Über ganze birationale Transformationen der Ebene.
J. Reine Angew. Math.184(1942), 161–174.

\bibitem[K93]{K93} Kaliman, Shulim
Smooth contractible hypersurfaces in $\mathbb{C}^{n}$ and exotic algebraic structures on $\mathbb{C}^{3}$.
Math. Z.214(1993), no.3, 499–509.

\bibitem[KML97]{KML97} SKaliman, Shulim; Makar-Limanov, Leonid
On the Russell-Koras contractible threefolds.
J. Algebraic Geom.6(1997), no.2, 247–268.

\bibitem[KKMLR97]{KKMLR97} Kaliman, S.; Koras, M.; Makar-Limanov, L.; Russell, P.
$\mathbb{C}^*$-actions on $\mathbb{C}^3$ are linearizable.(English summary)
Electron. Res. Announc. Amer. Math. Soc.3(1997), 63–71.

\bibitem[KZ99]{KZ99} Kaliman, Sh.; Zaidenberg, M.
Affine modifications and affine hypersurfaces with a very transitive automorphism group.
Transform. Groups4(1999), no.1, 53–95.

\bibitem[K79]{K79} Kambayashi, T.
Automorphism group of a polynomial ring and algebraic group action on an affine space.
J. Algebra60(1979), no.2, 439–451.

\bibitem[KR82]{KR82} Kambayashi, T.; Russell, P.
On linearizing algebraic torus actions.
J. Pure Appl. Algebra23(1982), no.3, 243–250.

\bibitem[K91]{K91} Knop, Friedrich
Nichtlinearisierbare Operationen halbeinfacher Gruppen auf affinen Räumen.
Invent. Math.105(1991), no.1, 217–220.

\bibitem[KR86]{KR86} Koras, Mariusz; Russell, Peter
$\mathbb{G}_{m}$-actions on $\mathbb{A}^{3}$.Proceedings of the 1984 Vancouver conference in algebraic geometry, 269–276.
CMS Conf. Proc., 6
Published by the American Mathematical Society, Providence, RI; for the, 1986
ISBN:0-8218-6010-0 
 

\bibitem[KR88.1]{KR88.1} Koras, Mariusz; Russell, Peter
Codimension 2 torus actions on affine n-space.Group actions and invariant theory (Montreal, PQ, 1988), 103–110.
CMS Conf. Proc., 10
Published by the American Mathematical Society, Providence, RI; for the, 1989
ISBN:0-8218-6015-1

\bibitem[KR88.2]{KR88.2} Koras, Mariusz; Russell, Peter
On linearizing "good" $\mathbb{C}^*$-actions on $\mathbb{C}^3$.Group actions and invariant theory (Montreal, PQ, 1988), 93–102.
CMS Conf. Proc., 10
Published by the American Mathematical Society, Providence, RI; for the, 1989
ISBN:0-8218-6015-1

\bibitem[KR97]{KR97}M. Koras, P. Russell. Contractible threefolds
and $\mathbb{C}^{*}$-actions on $\mathbb{C}^{3}$. J. Algebraic Geom.
6 (1997), no. 4, 671- 695.
\bibitem[K96]{K96}Kraft, Hanspeter
Challenging problems on affine n-space.(English summary)
Séminaire Bourbaki, Vol. 1994/95
Astérisque(1996), no.237, Exp. No. 802, 5, 295–317.

\bibitem[KP85]{KP85}Kraft, Hanspeter; Popov, Vladimir L.
Semisimple group actions on the three-dimensional affine space are linear.
Comment. Math. Helv.60(1985), no.3, 466–479.

\bibitem[KPR89]{KPR89}Kraft, Hanspeter; Petrie, Ted; Randall, John D.
Quotient varieties.
Adv. Math.74(1989), no.2, 145–162.

\bibitem[KS88]{KS88}Kraft, Hanspeter; Schwarz, Gerald W.
Reductive group actions on affine space with one-dimensional quotient.Group actions and invariant theory (Montreal, PQ, 1988), 125–132.
CMS Conf. Proc., 10
Published by the American Mathematical Society, Providence, RI; for the, 1989

\bibitem[KS92]{KS92}Kraft, Hanspeter; Schwarz, Gerald W.
Reductive group actions with one-dimensional quotient.
Inst. Hautes Études Sci. Publ. Math.(1992), no.76, 1–97.

\bibitem[LP17]{LP17} Liendo, Alvaro; Petitjean, Charlie
Smooth varieties with torus actions.(English summary)
J. Algebra490(2017), 204–218.

\bibitem[ML96]{ML96}Makar-Limanov, L.
On the hypersurface $x+x^2y+z^2+t^3=0$ in $\mathbb{C}^4$ or a $\mathbb{C}^3$-like threefold which is not C3. Israel J. Math.96(1996), 419–429.


\bibitem[MMJP91]{MMJP91} Masuda, Mikiya; Moser-Jauslin, Lucy; Petrie, Ted
Equivariant algebraic vector bundles over representations of reductive groups: applications.
Proc. Nat. Acad. Sci. U.S.A.88(1991), no.20, 9065–9066.

\bibitem[P83]{P83} Panyushev, D. I.
Semisimple groups of automorphisms of a four-dimensional affine space.
Izv. Akad. Nauk SSSR Ser. Mat.47(1983), no.4, 881–894.

\bibitem[P18]{P18} Petitjean, Charlie
Smooth contractible threefolds with hyperbolic Gm-actions via polyhedral divisors.
Manuscripta Math.156(2018), no.3-4, 399–408.

\bibitem[P15]{P15} Petitjean, Charlie
Cyclic covers of affine $\mathbb{T}$-varieties.
J. Pure Appl. Algebra219(2015), no.9, 4265–4277.

\bibitem[S89]{S89} Schwarz, Gerald W.
Exotic algebraic group actions.
C. R. Acad. Sci. Paris Sér. I Math.309(1989), no.2, 89–94.

\bibitem[S74]{S74} Propriétés topologiques des polynômes de deux variables complexes, et automorphismes algébriques de l'espace $\mathbb{C}^2$.(French)
J. Math. Soc. Japan26(1974), 241–257.

\end{thebibliography}
\end{document}